%% file: main.tex
\documentclass[]{amsart}

\input{preamble}

\usepackage[margin=1.5in]{geometry}
\usepackage{subcaption}

\title{Finite index rigidity of free-by-cyclic groups}
\author{Radhika Gupta}
\address{School of Mathematics, Tata Institute of Fundamental Research, Mumbai 40005, India}
\email{rgupta@math.tifr.res.in}

\author{Nir Lazarovich}
\address{Department of Mathematics,
	Technion,
	Haifa 32000,
	Israel}
 \email{lazarovich@technion.ac.il}
\date{}

\begin{document}

\begin{abstract} 
We show that a free-by-cyclic group with infinite order monodromy is finite index rigid, that is, any two isomorphic finite index subgroups have the same index.
\end{abstract}

\maketitle

\section{Introduction}

A group $G$ is called \emph{finite index rigid }(FIR) if any pair of isomorphic finite index subgroups of $G$ have the same index in $G$. Our main result is: 
\begin{theoremA}\label{main thm}
   Let $\bbF$ be a finite rank non-abelian free group and let $\phi\in \Out(\bbF)$.
   The free-by-cyclic group $\bbG \cong \bbF \rtimes_\phi \bbZ$ is finite index rigid if and only if $\phi$ has infinite order in $\Out(\bbF)$.
\end{theoremA}

Let $\phi \in \Out(\bbF)$. 
We say $\phi$ is \emph{polynomially growing} if there exists $d\in \{0\}\cup \bbN$  such that 
$\|\phi^n([g])\|\le O(n^d)$ for all $g\in \bbF$ (where $\|[g]\|$ denotes the minimum word length of elements of the conjugacy class $ [g]$ with respect to some fixed basis). The minimal such $d$ is the \emph{degree} $\deg(\phi)$ of $\phi$. We say that $\phi$ is \emph{exponentially growing} if there exists $g\in \bbF$ such that $\|\phi^n([g])\| \ge c^n$ for some $c>1$. By \cite{BestvinaHandel}, every $\phi\in \Out(\bbF)$ is either polynomially or exponentially growing.

If $\phi\in\Out(\bbF)$ has finite order (or equivalently if it is polynomially growing of degree 0), then $\bbG$ is virtually a direct product $\bbF\times \bbZ$ and hence not finite index rigid.
Recent work of Lazarovich--Rahamim--Sisto \cite{lazarovichrahamimsisto} addresses \Cref{main thm} in the case when $\phi$ is exponentially growing (see also \cite{lazarovichfir} for the atoroidal case). 
Thus, it remains to prove the following:

\begin{theoremA}\label{linearly growing are fir}
    Let $\bbF$ be a finite rank non-abelian free group. 
    Let $\phi\in \Out(\bbF)$ be a polynomially growing outer automorphism with $\deg(\phi) \geq 1$.
    Then the free-by-cyclic group $\bbG \cong \bbF \rtimes_\phi \bbZ$  is finite index rigid.
\end{theoremA}

By \cite{Dahmani-Li, Ghosh} (see also  \cite{JPPoly}), the growth of the monodromy $\phi$ is a quasi-isometry invariant of $\bbG$. 
In particular, if $\bbG$ is quasi-isometric to $\bbG' = \bbF' \rtimes_\psi \bbZ$, then $\phi$ is a polynomially growing outer automorphism if and only if $\psi$ is. Moreover, by \cite{Macura}, $\deg(\phi) = \deg(\psi)$. Therefore, we say  $\bbG$ is a \emph{polynomial free-by-cyclic group} with $\deg(\bbG) = d$ if $\phi$ is polynomially growing with $\deg(\phi) = d$.

\subsubsection*{Motivation from graph manifolds} Our approach to show that polynomial free-by-cyclic groups are finite index rigid is based on the corresponding result for graph manifolds. We give a brief overview of the graph manifolds case. 

A \emph{graph manifold} $M$ is a 3-manifold that admits a non-empty splitting surface (composed of tori) cutting $M$ into Seifert manifold pieces (called the vertex manifolds). 
In \cite{wangwu}, Wang--Wu exhibited a covering invariant for a closed Seifert 3-manifold $M$  given by $\chi(OM)^2/e(M)$, where $\chi(OM)$ is the Euler characteristic of the base orbifold and $e(M)$ is the \emph{Euler number}. This invariant is non-zero when $M$ is not covered by (surface)$\times S^1$ or a torus bundle over $S^1$. In the same paper, the authors defined a covering invariant $\omega(M)$ for a graph manifold $M$ that depends on the \emph{fiber intersection numbers} along the splitting tori and the orbifold Euler characteristics of the vertex manifolds. However, $\omega(M)$ can be zero for some graph manifolds. 

A straightforward way to define a covering invariant for $M$ would be to sum over covering invariants of the vertex manifolds. In \cite{lueckewu}, Luecke--Wu did exactly this. They defined a notion of \emph{relative Euler number} and used it to define a covering invariant for Seifert manifolds with boundary (the vertex manifolds). However, this covering invariant $\sigma(M)$ for the graph manifold can also be zero at times.

Based on the setup created by Wang--Wu,  in \cite{yuwang} Yu--Wang circumvented having to define an explicit covering invariant and directly showed that all graph manifolds, except those covered by (surface)$\times S^1$ or a torus bundle over $S^1$, are finite index rigid.

\subsubsection*{Proof strategy}
A polynomial free-by-cyclic group $\bbG$ of degree at least one admits a canonical \emph{tree of cylinders} (JSJ decomposition) with at least one vertex group a polynomial free-by-cyclic groups of lower degree (\Cref{canonical tree}). 
We inductively define a covering invariant for $\bbG$ by summing up covering invariants of vertex groups (of lower degree) of the tree of cylinders (\Cref{prop: multiplicative function for free-by-cyclic}). 
However, in the base case when $\bbG$ is a linear free-by-cyclic group, the vertex groups of the tree of cylinders are direct products of a free group with infinite cyclic group and don't satisfy finite index rigidity. 
In this case, we use the techniques of \cite{yuwang} to obtain finite index rigidity (\Cref{prop:fir linear}). 

\subsubsection*{Acknowledgments}
We would like to thank Naomi Andrew, Yassine Guerch, Camille Horbez and Richard D. Wade, the authors of \cite{andrewguerchhorbezwade}, for sharing a  preprint of their ongoing work. We also thank Naomi Andrew for comments on a draft. 
We gratefully acknowledge the hospitality of the Isaac Newton Institute and Tata Institute of Fundamental Research which enabled this collaboration.
RG was supported by the Department of Atomic Energy, Government of India, under Project Identification No. RTI 4014 
and partially supported by an endowment of the Infosys Foundation.
NL was supported by the Israel Science Foundation (grant no. 1576/23).

\section{Preliminaries on Finite index rigidity}
In this section we prove some general statements about finite index rigidity and multiplicative functions on a collection of groups.

\subsection{Finite index rigidity}
A group $G$ is called \emph{finite index rigid } (FIR) if for any pair of finite index subgroups $H_1, H_2 \le G$, if $H_1 \cong H_2$ then $[G:H_1] = [G:H_2]$.

A group $G$ is called \emph{normal finite index rigid} (nFIR) if for any normal finite index subgroups $N_1,N_2\normalin G$, $N_1 \cong N_2$ implies $[G:N_1]=[G:N_2]$. The following lemma is an adaptation of {\cite[Lemma 2.1]{wangwu}}.

\begin{lemma}\label{reduction to normal}
The following are equivalent:
\begin{enumerate}
    \item $G$ is FIR.
    \item Every finite index subgroup of $G$ is FIR.
    \item Every finite index subgroup $H\le G$ has a finite index subgroup $K\le H$ that is nFIR.
\end{enumerate}
\end{lemma}
\begin{proof}
The implications $(1)\implies (2)\implies (3)$ are clear.
To prove $(3)\implies (1)$, let $H_1,H_2\le G$ be finite index subgroups, and let $f:H_1\to H_2$ be an isomorphism. 
By the assumption there exists a finite index subgroup $K_2\le H_2$ which is nFIR.
Let $K_1=f\ii(K_2)$.
Then $[H_1:K_1]=[H_2:K_2]$.
Let $N\le K_1\cap K_2$ be a normal finite index subgroup of $G$. Then $N\normalin K_2$, and $f(N)\normalin K_2$ (because $N\normalin K_1$).
Since $K_2$ is nFIR, $[K_2:N]=[K_2:f(N)]=[K_1:N]$. 
Finally, using $$[G:H_1][H_1:K_1][K_1:N]=[G:N]=[G:H_2][H_2:K_2][K_2:N]$$ we conclude that $[G:H_1]=[G:H_2]$.
\end{proof}

\subsection{Multiplicative functions}
We now define a multiplicative function on a set of groups and use it to define a generalized Euler characteristic. 

\begin{definition}
Let $\calG$ be a set of groups closed under taking finite index subgroups. A \emph{multiplicative function} $\xi\colon\calG\to \bbR$ is a 
function that satisfies:
\begin{enumerate}
    \item $\xi$ is a group invariant, i.e. if $G_1\simeq G_2$ then $\xi(G_1)=\xi(G_2)$.
    \item $\xi$ is proportional to index, i.e. if $H\le G$ has finite index then $\xi(H)=[G:H]\xi(G)$.
\end{enumerate}
\end{definition}

We say two groups $H_1$ and $H_2$ are \emph{abstractly commensurable} if there exist finite index subgroups $K_1 \le H_1$ and $K_2 \le H_2$, such that $K_1 \cong K_2$. This is an equivalence relation on the set of groups. 

\begin{remark}
    The two conditions in the above definition imply that if $H_1,H_2$ are abstractly commensurable, then $\xi(H_1)$ is determined by $\xi(H_2)$. Thus, $\xi$ can be defined independently on abstract commensurability classes. 
\end{remark}

\begin{lemma}\label{fir to multiplicative}
    Let $\calG$ be a set of groups closed under taking finite index subgroups. There exists a multiplicative function $\xi\colon \calG \to \bbRnn$ such that $\xi(G)>0$ for every finite index rigid group.
\end{lemma}

\begin{proof}
    If $G$ is finite index rigid then any group commensurable to $G$ is finite index rigid. 
    For each commensurability class of groups in $\calG$, pick a representative and set the function $\xi$ on this representative to be $1$ if it is finite index rigid and $0$ otherwise.
    For any group $G\in \calG$, set $\xi(G)=\frac{[G':K']}{[G:K]}\xi(G')$ where $G'$ is the representative of the commensurability class of $G$, and $K,K'$ are isomorphic finite index subgroups in $G,G'$ respectively.
\end{proof}

Inspired by the work of Sykiotis \cite{sykiotis2018complexity}, we define a generalized Euler characteristic for a group acting on a set and show that it behaves well with respect to a finite index subgroup.

\begin{definition}[Generalized Euler characteristic]
Let $\calG$ be a set of groups closed under taking finite index subgroups. Let $G$ be a group acting on a set $X$ with finitely many orbits and stabilizers in $\calG$. Let $\xi\colon \calG\to \bbR$ be a multiplicative function.
Define
$$\chi_\xi(G\actson X) = \sum_{\cl{x}\in X/G} \xi(G_{\cl{x}}), $$ where $G_{\cl{x}}$ denotes the stabilizer in $G$ of a representative of $\cl{x}$ in $X$.  
\end{definition}

\begin{lemma}\label{sum of indices}
    Let $G\actson  X$ transitively and let $H\le G$ be a finite index subgroup.  Let $x_1,\dots,x_k\in X$ be $H$-orbit representatives with stabilizers $H_{x_i}$ (resp. $G_{x_i}$) in $H$ (resp. $G$), then $\sum _{i=1}^k[G_{x_i}:H_{x_i}]=[G:H].$
\end{lemma}

\begin{proof}
    Fix $x_0\in X$, and fix $g_i\in G$ such that $g_i x_0=x_i$.
    Fix  
    $g_{i1},\dots,g_{id_i}$ 
    representatives of right $H_{x_i}$-cosets in $G_{x_i}$ (where $d_i=[G_{x_i}:H_{x_i}]$).
    Let us show that the set $\{g_{ij}g_i\mid 1\le i\le k,1\le j\le d_i\}$, which has size $\sum_{i=1}^k d_i$ is a set of right $H$-coset representatives in $G$.

    For all $g\in G$, there exists a unique $i$ and some $h\in H$ such that $gx_0=hx_i=gg_i^{-1}x_i$, from which we get $h\ii g g_i\ii \in G_{x_i}$. There exists a unique $g_{ij}$ such that $h\ii g g_i\ii=h'g_{ij}$ for $h'\in H_{x_i}$.
    Therefore, $g=hh'g_{ij}g_i$ is contained in the right $H$-coset $Hg_{ij}g_i$.

    Suppose $Hg_{ij}g_i = Hg_{kl}g_k$.   
     Since $g_{ij}g_i x_0 = x_i$ and $g_{kl}g_k x_0 = x_k$, we get $H x_i=Hg_{ij}g_i x_0 = Hg_{kl}g_k x_0=H x_k$. which implies  $i=k$,  $g_i=g_k$
    and $Hg_{ij}=Hg_{il}$. 
Thus $g_{ij}g_{il}^{-1} \in H$ and also by definition $g_{ij}g_{il}^{-1} \in G_{x_i}$. Therefore, $g_{ij}g_{il}^{-1} \in H_{x_i} = H \cap G_{x_i}$ which implies $H_{x_i}g_{ij}=H_{x_i}g_{il}$ and hence $j=l$.  
\end{proof}

\begin{corollary}\label{cor:gen Euler char}
Let $\calG$ be a set of groups closed under taking finite index subgroups, and let $\xi:\calG\to \bbR$ be a multiplicative function. Let $G$ be a group acting on a set $X$ with finitely many orbits and stabilizers in $\calG$.
If $H\le G$ has finite index, then 
$\chi_\xi(H\actson X)=[G:H]\chi_\xi(G\actson X)$.
\end{corollary}

\begin{proof}
    Let $v\in X/G$ and let $v_1,\dots,v_k\in X/H$ be the elements that project to $v$ under the quotient map $X/H\to X/G$. Then by \Cref{sum of indices} and multiplicativity of the function $\xi$, we get
    $$[G:H]\xi(G_v)=\sum _{i=1}^k[G_{v}:H_{v_i}]\xi(G_v)=\sum _{i=1}^k\xi(H_{v_i}).$$
    Summing over all vertices of $X/G$ gives the desired equality.
\end{proof}

We will use the above result in \Cref{sec:higher degree} to show that polynomial free-by-cyclic groups of degree greater than or equal to two are finite index rigid. 
\subsection{Reduction to UPG}

A polynomially growing free group outer automorphism $\phi$ is called \emph{unipotent polynomially growing} (UPG) if the induced automorphism of $H_1(\bbF, \bbZ)$ is unipotent. We finish this section with a reduction to polynomial free-by-cyclic groups with UPG monodromy.

\begin{lemma}\label{reduction to UPG and normal}
    Let $d\ge 1$. The following are equivalent:
    \begin{enumerate}
    \item For every UPG element $\phi$ of degree $d$, the group $\bbF\rtimes _\phi\bbZ $ is nFIR.
    \item For every polynomially growing $\phi$ of degree $d$, the group $\bbF\rtimes _\phi\bbZ$ is FIR.
    \end{enumerate}
\end{lemma}

\begin{proof}
    Clearly, (2) implies (1).
    Now assume (1) holds. 
    By \Cref{reduction to normal} it suffices to show that every finite index subgroup has a further finite index subgroup which is nFIR.
    Indeed, every finite index subgroup is a polynomial free-by-cyclic group, by \cite[Proposition 3.5]{BFH:TitsII} every such group has a finite index subgroup which is free-by-cyclic with UPG monodromy, and by assumption those are nFIR.
\end{proof}

\section{Canonical trees}
The first goal in our proof of finite index rigidity is to find a canonical action $\bbG\actson \bbT$ with abelian edge stabilizers:

\begin{theorem}[{\cite{andrewguerchhorbezwade} and \cite{andrew2022free}}]\label{canonical tree}
    Let $\bbG=\bbF \rtimes_\phi \bbZ = \gen{\bbF,t  \mid  tgt\ii =\Phi(g)}$ where $\phi$ is a UPG outer automorphism of degree $d$ and $\Phi \in \Aut(\bbF)$ represents $\phi$.
    
    Then there exists a tree $\bbT = \bbT_\bbG$ and a minimal action $\bbG\actson \bbT$ such that:
    \begin{enumerate}[label = (T\arabic*)]
        \item\label{T: vertex groups} 
        Vertex groups are polynomial free-by-cyclic \footnote{Here \emph{free-by-cyclic} includes degenerate cases, such as $\bbZ=1\rtimes \bbZ$ and $\bbZ^2=\bbZ\rtimes \bbZ$, which are clearly polynomially growing of degree $0$.} groups of degree at most $d-1$, with at least one vertex group of degree equal to $d-1$. 
        \item \label{T: finite index}
        If $H\le \bbG$ has finite index, then $\bbT_H$ is $H$-equivariantly isomorphic to $\bbT_G$.
        \item\label{T: commensurability invariance} If $\eta\colon H\to H'\le \bbG$ is an isomorphism between finite index subgroups $H,H'\le \bbG$, then there is an $\eta$-equivariant isomorphism $H\actson \bbT_\bbG\to H'\actson \bbT_\bbG$. 
    \end{enumerate}
    If $d=1$, then moreover
    \begin{enumerate}[label = (T\arabic*)]\setcounter{enumi}{3}
        \item\label{T: edge groups linear} Edge groups are maximal $\bbZ^2$ subgroups of $\bbG$.
        \item\label{T: vertex groups linear} The tree has a bipartition $V(\bbT) = V_0 \sqcup V_1$. The stabilizers of vertices in $V_0$ are of the form $F \times \gen{yt}$ where $F\le \bbF$ is non-abelian and $y\in \bbF$; and those of vertices in $V_1$ are isomorphic to $\bbZ^2$. 
        \item\label{T: cylinders linear} Every vertex in $V_1$ has finite valence.
    \end{enumerate}
\end{theorem}

    An analogous description of the edges and vertices of $\bbT$ holds for $d\ge 2$ (see \cite{andrewguerchhorbezwade}) however this will not be used in this paper.

\begin{proof}[Sketch of proof:]
First consider the limiting tree $\bbF\actson T_\phi$ given by 
\cite[Theorem 13.2]{cohen1995verysmall} when $d=1$ 
and \cite{BFH:TitsII} when $d \ge 2$ (also see \cite[Theorem 6.1]{andrewguerchhorbezwade} for details).
The action $\bbF\actson T_\phi$ has the following properties:
\begin{itemize}
    \item edge stabilizers of $\bbF \actson T_\phi$ are maximal infinite cyclic if $d=1$ and trivial if $d\ge 2$,
    \item vertex stabilizers of $\bbF \actson T_\phi$ are either maximal $\phi$-invariant subgroups $F$ of $\bbF$ of rank at least 2 such that $\deg(\phi|_F) < \deg(\phi)$, or trivial when $d \ge 2$, 
    \item there exists at least one vertex group $F$ as above with $\deg(\phi|_F) = d-1$.
\end{itemize}

The tree $T_\phi$ is $\phi$-invariant, i.e. there exists a $\Phi$-equivariant graph automorphism $f\colon T_\phi\to T_\phi$. 
The action $\bbF\actson T_\phi$ can be extended to an action $\bbG\actson T_\phi$ by setting $t \cdot x = f(x)$ (see \cite[Lemma 3.1]{JPPoly}). 
The action $\bbG\actson T_\phi$ has the following properties (see \cite[Lemma 6.7, 6.8]{andrewguerchhorbezwade}):
\begin{itemize}
    \item the edge stabilizers of $\bbG \actson T_\phi$ are maximal $\bbZ^2$ subgroups if $d=1$, and maximal cyclic if $d\ge 2$,
    \item vertex stabilizers of $\bbG \actson T_\phi$  are either maximal polynomial free-by-cyclic subgroups of the form $F \rtimes_\psi \bbZ$ for a free group $F$ of rank at least 2 and $\deg(\psi) < \deg(\phi)$, or $\bbZ$ when $d \geq 2$, 
    \item at least one vertex group is of the form $F \rtimes_\psi \bbZ$ with $\deg(\psi) = d-1$. 
\end{itemize}  

Finally, let $\bbT$ be the collapsed tree of cylinders of $T_\phi$ \`a la Guirardel-Levitt \cite{guirardel2011trees} with respect to the collection $\calE$ of maximal abelian subgroups and the admissible relation of commensurability \cite[Definition 6.13]{andrewguerchhorbezwade}.

Property \ref{T: vertex groups}, \ref{T: finite index} follow from \cite[Theorem 6.15, Proposition 6.18]{andrewguerchhorbezwade}, respectively.

Property \ref{T: commensurability invariance} is as follows: With respect to the action $\rho$ of $\bbG$ on the limiting tree, $T_\phi$ is a slender JSJ tree relative to free-by-cyclic subgroups of lower degrees and with $\bbZ$ or $\bbZ^2$ (depending on $d$) edge stabilizers by \cite[Proposition 6.12]{andrewguerchhorbezwade}. Therefore, the same is true for the restriction $\rho|_H$ and $\rho|_{H'}\circ \eta$. It follows from Property \ref{T: finite index} and \cite[Corollary 4.10]{guirardel2011trees} that an $\eta$-equivariant isomorphism $\bbT\to \bbT$ exists.

Property \ref{T: edge groups linear} and \ref{T: vertex groups linear} follow from \cite[Lemma 6.16]{andrewguerchhorbezwade} for UPG $\phi$ (see also \cite[Proposition 5.2.2]{andrew2022free}).

Property \ref{T: cylinders linear} is as follows: By maximality of edge stabilizers \ref{T: edge groups linear}, the stabilizer of each $v\in V_1$ is equal to the stabilizer of each of its incident edges. This shows that the valence of $v$ in $\bbT$ is equal to the valence of the corresponding vertex $\bar v$ in the quotient graph $\bbT/\bbG$. Since $\bbG$ is finitely generated, and $\bbG\actson \bbT$ is minimal, the graph $\bbT/\bbG$ is finite. In particular, $\bar v$ has finite valence, and so does $v$.
\end{proof}

We give an example of a linear UPG outer automorphism and the associated limiting tree, which we will use to illustrate the construction in the next section.  

\begin{example}\label{eg:limiting tree}
    Let $\bbF = \gen{a, b}$ and $\phi \in \Out(\bbF)$ be the outer class of  $ \Phi \colon a \mapsto a, b \mapsto ba$. Let \[\bbG = \bbF \rtimes_\phi \bbZ = \gen{a, b, t \mid  tat\ii = a, tbt\ii = ba}.\] 
    The tree $\bbF \actson T_\phi$ is obtained by the construction of the limiting tree explained in \cite[Lemma 4.37]{BFH:TitsII}. 
   In \Cref{fig:Example1} the left figure shows $T_\phi/\bbF$ with graph of groups edge inclusion maps given by 
    $\bbZ \to\gen{b\ii ab}$, $\bbZ \to \gen{a}$ 
    and the right figure shows $T_\phi/\bbG$ with 
    graph of groups edge inclusion maps
    $\bbZ^2 \to \gen{b\ii ab, at}$, $\bbZ^2 \to \gen{a, batb\ii} = \gen{a, t}$.
    The tree of cylinders $\bbT$ is simply the barycentric subdivision of $T_\phi$.

\begin{figure}[htbp]
  \centering
    \includegraphics[scale=.6]{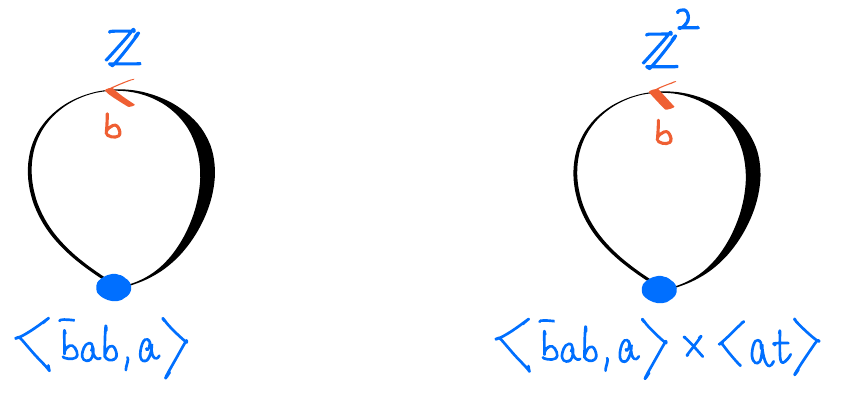}
    \caption{$T_\phi/\bbF$ and $T_\phi/\bbG$ }
    \label{fig:Example1}
\end{figure}
\end{example}

\section {Finite index rigidity}
In this section we prove that a polynomial free-by-cyclic group $\bbG$ is finite index rigid, first for $\deg(\bbG) = 1$ and then for $\deg(\bbG) \geq 1$. 

\subsection{Finite index rigidity for linear degree} \label{sec:fir linear}

 In this subsection, $\bbG = \bbF \rtimes_\phi \bbZ$ is a polynomial free-by-cyclic group with $\deg(\bbG) = 1$. We use the methods of Wang--Wu \cite{wangwu} and Yu--Wang \cite{yuwang} for graph manifolds to prove finite index rigidity for $\bbG$. 

Recall that the tree of cylinders $\bbT$ associated to 
$\bbG$ admits a bipartition with $V(\bbT) = V_0\sqcup V_1$ where stabilizers of vertices in $V_1$ are isomorphic to $\bbZ^2$. We construct a graph $\bfT$ by replacing each star of $v\in V_1$ by a complete graph on its neighbors.
By \ref{T: cylinders linear}, these stars and complete graphs are finite.
Thus, $V(\bfT) = V_0$ and two vertices are connected by an edge if they are connected to the same vertex in $V_1$. 
The edge stabilizers in $\bfT$ are again maximal $\bbZ^2$ and it is canonical as in \Cref{canonical tree}.

Let $\Gamma(\bbG) = \bfT/\bbG$. By Property \ref{T: vertex groups linear} of \Cref{canonical tree}, for each vertex $v\in \bfT/\bbG$ the vertex group $G_{v}$  is of the form $F \times \gen{y t}$ for some non-abelian subgroup $F\le \bbF$ and $y\in \bbF$. 
Define $\bfx(v)=\chi(F)=1-\rank(F)$.

We assign a number $\Delta(e)$ for an edge $e$ in $\Gamma(\bbG)$ as follows:
Let $\tild e$ be a lift of $e$ to $\bfT$, and let $v_1,v_2$ be its endpoints.
Then $G_{v_i}$ is of the form $F_i \times \gen{y_i t}$ for some (non-abelian) subgroup $F_i \le \bbF$ and $y_i\in \bbF$.
Note that $Z(G_{v_i}) = \gen{y_i t}$.
The subgroup $G_{\tild e}$ intersects $Z(G_{v_i})$ in an infinite cyclic subgroup $Z_i = G_{\tild e} \cap Z(G_{v_i})$.

\begin{lemma}
    The cyclic subgroups $Z_1,Z_2$ are spanned by independent vectors in $G_{\tild e}$.
\end{lemma}
\begin{proof}
Assume for the sake of contradiction that $Z_1,Z_2$ are dependent. 
Since $G_{\tild{e}}$ is a maximal $\bbZ^2$ subgroup of $\bbG$, we have $Z_i = \gen{y_i t}$, which implies that $Z_1=Z_2$. 
But then $G_{v_1}=G_{v_2}$ since they are the centralizers of $Z_1,Z_2$ respectively.
However, this shows that $G_{\tild e}$ contains $G_{v_1}$ which contradicts that $G_{\tild e}\simeq \bbZ^2$. 
\end{proof}

\begin{definition}[Intersection number]
Let $Z_1,Z_2\le\bbZ^2$ be two cyclic subgroups spanned by independent vectors. Define their intersection number $i(Z_1,Z_2) =  [\bbZ^2:(Z_1+Z_2)]$, or as a determinant with respect to a fixed basis.  
Equivalently, $i(Z_1,Z_2)$ is the (geometric) intersection number of the curves corresponding to $Z_1,Z_2$ when thinking of $\bbZ^2$ as the fundamental group of a torus.
\end{definition}

For an edge $e$ of $\Gamma (\bbG)$, define $\Delta(e)=i(Z_1,Z_2)$ and set the edge weights $$\alpha(e) = \frac{1}{\Delta(e)\bfx(v_1)\bfx(v_2)}.$$ 
Similarly, for a finite index subgroup $H < \bbG$, define $\Gamma(H) = \bfT/H$ and weights on its edges as above. There is a natural quotient map $p \colon \Gamma(H) \to \Gamma(\bbG)$. 

\begin{lemma}\label{isomorphic covers}
    If $\eta \colon H\to H'$ is an isomorphism between finite index subgroups, then the $\eta$-equivariant map $\Theta\colon \bbT\to \bbT$ of \Cref{canonical tree} induces a weight-preserving graph isomorphism $\bar \Theta\colon \Gamma(H)\to \Gamma(H')$. \hfill  \qedsymbol
\end{lemma} 

Here is a simple example to illustrate the intersection number.

\begin{example}
Recall the setup from \Cref{eg:limiting tree}. 
Note that $\bfT$ is isomorphic to $T_\phi$ in this example. 
Let $A < \bbF$ be an index 2 subgroup given by $\gen{ a^2,b, aba\ii}$. 
In \Cref{fig:Example2}, the left figure shows $\bfT/A$.

   Note that $A$ is not $\Phi$-invariant but it is $\Phi^2$-invariant. 
Therefore, the subgroup $H=\gen{A,t^4} <\bbG$ has index 8 in $\bbG$ and has the form 
$H = A \rtimes_{\Phi^4|_A} \bbZ$ with monodromy represented by $t^4$. In \Cref{fig:Example2}, the right figure shows $\bfT/H$ with vertex $\hat v$ and edge $\hat e$. The 
graph of groups inclusion maps for $\hat e$ are $\bbZ^2 \to \gen{b\ii a^2b, a^4t^4}$ and $\bbZ^2 \to \gen{a^2, t^4}$. Let $\tild e$ be a lift of  $\hat{e}$ with stabilizer $\gen{a^2, t^4}$ and with end points $v_1, v_2$. Then $Z_1 = \gen{a^4t^4}$ and $Z_2 = \gen{t^4}$.  Using $e,v$ to denote the edge and vertex of $\bfT/\bbG$ shown in \Cref{fig:Example1}, we get

\begin{align*}
     \Delta(e) & = [\gen{a,t} : \gen{t}+ \gen{at} ] = 1  & \bfx(v)&=-1 & \alpha(e)&=\frac{1}{\Delta(e)\bfx(v)^2}=1\\
     \Delta(\hat e) & = [\gen{a^2,t^4} : \gen{t^4}+ \gen{a^4 t^4} ] = 2 & \bfx(v')&=-2  & \alpha(\hat e)&=\frac{1}{\Delta(e_1)\bfx(v')^2}= \frac{1}{8}
    \end{align*}     

\begin{figure}[htbp]
  \centering
    \includegraphics[scale=.6]{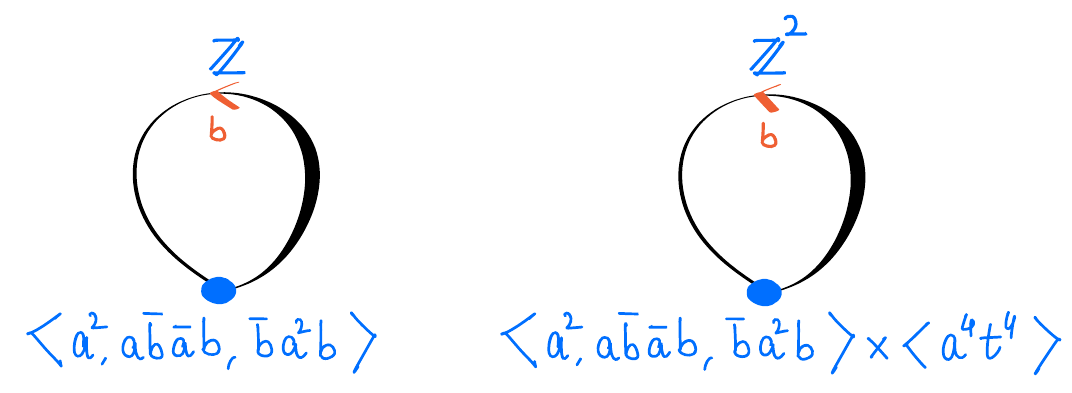}
    \caption{$\bfT/A$ and $\bfT/H$ }
    \label{fig:Example2}
\end{figure}
\end{example}
The following lemma is a rewording of \cite[Proposition 5.1]{wangwu} in group theoretical language. We include a proof as this is the basis of the Wang--Wu and Yu--Wang machinery. 

\begin{lemma}\label{alpha of cover}
    Let $\hat G\normalin \bbG$ be a normal subgroup of index $d$. Denote by $p\colon \Gamma(\hat G)=\bfT/\hat G\to \bfT/\bbG=\Gamma(\bbG)$ the quotient map. 
    Let $e$ be an edge of $\Gamma(\bbG)$ with endpoints $v_1,v_2$ and let $\hat e$ be an edge of $\Gamma(\hat G)$ with $p(\hat e)=e$. Let $u_i = |p\ii (v_i)|$ and $u_e = |p\ii(e)|$.  Then
    \begin{equation}
        \label{desired equality} d u_e\alpha(\hat e) = \alpha(e)u_1u_2.
    \end{equation}
\end{lemma}

\begin{proof}
    Let $\tilde e$ be an edge of $\bfT$ which lifts $e$ and $\hat e$, and let $\tilde v_1$ and $\tilde v_2$ be the end points of $\tilde e$. 
    Denote by $G_i = \bbG_{\tilde v_i}$ (resp. $\hat G_i = \hat G_{\tilde v_i}$) the stabilizer of $\tilde v_i$ in $\bbG$ (resp. $\hat G$) . 
    Then, since $\hat G\normalin \bbG$, \Cref{sum of indices} implies
    \begin{equation}\label{eq1}
        d=[\bbG:\hat G]=[G_i:\hat G_i] u_i.
    \end{equation}
    Recall that $G_i\simeq F_i \times Z_i$ where $F_i$ is free of rank at least 2 and $Z_i = Z(G_i)\simeq \bbZ$. 
    Similarly, $\hat G_i=\hat F_i\times \hat  Z_i$ where $\hat F_i = F_i \cap \hat G_i$ and $\hat Z_i$ is the image of $\hat G_i$ under the projection $G_i \epi Z_i$.
    Then, 
    \begin{equation}
        \label{eq2}[G_i:\hat G_i]=[Z_i:\hat Z_i][F_i:\hat F_i]
    \end{equation}
    Since the Euler characteristic is proportional to index, we get for endpoints $\hat v_i$ of $\hat e$,
    \begin{equation}
        \label{eq3}\bfx(\hat v_i)=\chi(\hat F_i) = [F_i:\hat F_i]\chi(F_i) = [F_i:\hat F_i]\bfx(v_i).
    \end{equation}
    Similarly, denote $E = G_{\tilde e}$ and $\hat E = \hat G_{\tilde e}$, the stabilizers of $\tilde e$ in $\bbG$ and $\hat{G}$, respectively.
    Then,
    \begin{equation}
        \label{eq4} d=[E:\hat E] u_e
    \end{equation}
    Finally, since $\Delta(e) = [E:Z_1+Z_2]$ and $\Delta(\hat e) = [\hat E:\hat Z_1+\hat Z_2]$, we have
    \begin{equation}
        \label{eq5} [Z_1:\hat Z_1][Z_2:\hat Z_2]\Delta(e)=[E:\hat Z_1+\hat Z_2]=\Delta(\hat e)[E:\hat E]
    \end{equation}

    Combining \Cref{eq1,eq2,eq3,eq4,eq5} gives the desired equality \eqref{desired equality}.
\end{proof}

Finally, we show that a polynomial free-by-cyclic group $\bbG$ of degree 1 is finite index rigid. It is a direct application of the methods of Yu--Wang \cite{yuwang} so we include the setup but omit the proofs. 

\begin{proposition}\label{prop:fir linear}
        Let $\bbF$ be a finite rank non-abelian free group. 
    Let $\phi\in \Out(\bbF)$ be a linearly growing outer automorphism.
    Then the free-by-cyclic group $\bbG = \bbF \rtimes_\phi \bbZ$ is finite index rigid.
\end{proposition}

\begin{proof}
    By \Cref{reduction to UPG and normal}, it suffices to show that $\bbG$ is normal finite index rigid when $\phi$ is a linear UPG outer automorphism.
    Let $H$ be a normal finite index subgroup of $\bbG$. 
    Define a non-negative real-valued symmetric matrix $\bfB^\bbG$ indexed by $V(\Gamma(\bbG))^2$ to be
    \[\bfB^\bbG_{u,v} = \begin{cases}
        \sum_{e} \alpha(e) &u\ne v\\
        2\sum_{e} \alpha(e) &u= v
    \end{cases} \]
    where the sum runs over all edges $e$ with endpoints $u,v$. Similarly, one can define a matrix $\bfB^H$ using $\Gamma(H)$. However, there is some symmetry in $\bfB^H$ and we will define a new matrix $\bfH^H$ obtained from $\bfB^H$ by an averaging process. 
    
    Let $\Aut(\Gamma(H))$ be the group of automorphisms of $\Gamma(H)$ preserving the weights.
    Consider $O(H) = V(\Gamma(H))/\Aut(\Gamma(H))$ and denote by $|a|$ the number of vertices in an orbit $a$. The symmetry in $B^H$ is as follows (see \cite[Theorem 2]{yuwang}): for two orbits $a, b \in O(H)$, the interaction of two vertices $p \in a$ and $q \in a$ with the orbit $b$ is exactly the same in the sense that 
    \[ \sum_{s \in b} \bfB^H_{p,s} = \sum_{s \in b} \bfB^H_{q,s}.\]
    
    Motivated by this symmetry, define a non-negative real-valued symmetric matrix $\bfH^H$ indexed by $O(H)^2$ to be
\[\bfH^H_{a,b}=\frac{1}{|a||b|}\sum_{u\in a} \sum_{v\in b} \bfB^H_{u,v} \]

    Let $p\colon \Gamma(H)\to \Gamma(\bbG)$ be the quotient map. 
    For $a \in O(H)$ and $u \in V(\Gamma(\bbG))$, define \[\bfX_{a,u} ^H= \begin{cases} |p\ii(u)|/|a| & p\ii(u) \subseteq a\\ 0 &p\ii(u) \nsubseteq a\end{cases} \quad \text{ and }\quad \bfW^H_{a,u} =\begin{cases} 1 &p\ii(u) \subseteq a\\ 0 &p\ii(u) \nsubseteq a\end{cases},  \]
    where $\bfX^H$ captures normalized sizes of fibers under the map $p$ and $\bfW^H$ records orbit membership. 
    Then, as in the proof of  \cite[Theorem 3]{yuwang}, the equation from \Cref{alpha of cover} that captures the information of the quotient map $p$, can be converted into the matrix equation 
        \[\bfX^H   \bfB^\bbG=[\bbG:H] \, \bfH^H  \bfW^H.\]

    Now let $H,H'$ be normal finite index subgroups of $\bbG$, and let $\eta\colon H\to H'$ be an isomorphism.
    By \Cref{isomorphic covers}, there is a weight-preserving graph isomorphism $\bar \Theta\colon \Gamma(H)\to \Gamma(H')$.
    Therefore, $\bfH^H$ and $\bfH^{H'}$ are the same up to re-indexing given by $\bar \Theta$. In other words, $\bfH^H$ is independent of the quotient maps $p$. However, note that both $\bfX^H$ and $\bfW^H$ depend on $p$.  
    As in the proof of \cite[Theorem 4]{yuwang}, by manipulating the two equations
  \begin{align*}
       \bfX^H   \bfB^\bbG & =[\bbG:H] \, \bfH^H\bfW^H \\ 
       \bfX^{H'}   \bfB^\bbG & =[\bbG:H'] \, \bfH^{H'} \bfW^{H'}
   \end{align*} 
    and an application of Perron--Frobenius theorem imply that $[\bbG:H]=[\bbG:H']$.
\end{proof}

\subsection{Finite index rigidity in higher degree}
\label{sec:higher degree}
 In this subsection, $\bbG = \bbF \rtimes_\phi \bbZ$ is a polynomial free-by-cyclic group with $\deg(\bbG) \geq 1$. 
 Theorem \ref{linearly growing are fir} follows from the following:

\begin{proposition}\label{prop: multiplicative function for free-by-cyclic}
Let $d\ge 1$, and let $\calG_d$ be the collection of all polynomial free-by-cyclic groups of degree $d$. There exists a multiplicative function $\xi_d \colon \calG_d \to (0,\infty)$.
\end{proposition}

Note that $\calG_d$ is closed under taking finite index subgroups.

\begin{proof}
We prove the proposition by induction on $d$.
Every polynomial free-by-cyclic group $\bbG$ with $\deg(\bbG) =1$ is finite index rigid by \Cref{prop:fir linear} and hence, by \Cref{fir to multiplicative}, there exists a multiplicative function $\xi_1\colon \calG_1\to (0,\infty)$.

Suppose for every $1 \leq i \leq d$ we have a multiplicative function $\xi_i\colon \calG_i\to (0,\infty)$. 
Let $\calG_0$ be the collection of all polynomial free-by-cyclic groups of degree $0$ (including degenerate free-by-cyclic groups such as $\bbZ$ and $\bbZ^2$).
Let $\xi_0 \colon \calG_0\to [0,\infty)$ be the constant $0$ function and let $\xi \colon \bigcup_{i=0}^d\calG_i\to [0,\infty)$ be defined by $\xi = \bigcup _{i=0}^d\xi_i$. 

Let $\bbG$ be a polynomial free-by-cyclic group with $\deg(\bbG) = d+1>1$. 
Let $\bbG\actson \bbT_\bbG$ be the canonical tree given by \Cref{canonical tree}. 
By \ref{T: vertex groups}, all vertex groups are in $\bigcup_{i=0}^d\calG_i$ therefore $\xi(G_v)$ is defined for every vertex group $G_v$.
Moreover, at least one vertex group $G_{v_0}$ is contained in $\calG_d$ which implies $\xi(G_{v_0})>0$. 
Therefore,
\[\chi_\xi(\bbG \actson \bbT_\bbG) = \sum_{v \in \bbT/\bbG} \xi(G_v) >0.\]

We define $\xi_{d+1}(\bbG):= \chi_\xi(\bbG \actson \bbT_\bbG) >0$. Since $\bbT_\bbG$ is canonical for $\bbG$ (in the sense of \ref{T: commensurability invariance}), $\xi_{d+1}(\bbG)$ is an isomorphism invariant.  
Moreover, by \ref{T: finite index} if $H$ is a finite index subgroup of $\bbG$, then $H \actson \bbT_H$ is equivariantly isomorphic to $H\actson \bbT_\bbG$. By \Cref{cor:gen Euler char}, $$\xi_{d+1}(H) = \chi_\xi(H\actson \bbT_H) = \chi_\xi(H\actson \bbT_\bbG) = [\bbG:H]\chi_\xi(\bbG \actson \bbT_\bbG) = [\bbG:H]\xi_{d+1}(\bbG).$$ 
Thus, $\xi_{d+1}$ is a non-zero multiplicative function on $\calG_{d+1}$. 
\end{proof}

\bibliographystyle{alpha}
\bibliography{biblio}
\end{document}

%% file: preamble.tex
\usepackage[leqno]{amsmath}
\usepackage{amssymb,amsfonts,xfrac,MnSymbol}
\usepackage{amsthm}
\usepackage{thmtools} 
\usepackage{cite}
\usepackage{hyperref}
\usepackage{todonotes}
\usepackage{enumitem}
\usetikzlibrary{calc}
\usepackage{graphicx}
\usepackage{tikz-cd, tikz}
\usepackage{color}
\usepackage{import}
\usepackage{cleveref}

\numberwithin{equation}{section}

\usepackage{marginnote}
\usepackage{todonotes}

\newtheorem{theorem}{Theorem}[section]

\newtheorem{proposition}[theorem]{Proposition}
\newtheorem{lemma}[theorem]{Lemma}
\newtheorem{corollary}[theorem]{Corollary}

\newtheorem*{theorem*}{Theorem}
\newtheorem*{claim*}{Claim}
\newtheorem*{proposition*}{Proposition}
\newtheorem*{lemma*}{Lemma}
\newtheorem*{corollary*}{Corollary}

\newtheorem{theoremA}{Theorem}

\theoremstyle{definition}
\newtheorem{definition}[theorem]{Definition}

\newtheorem{remark}[theorem]{Remark}
\newtheorem{example}[theorem]{Example}

\newtheorem*{definition*}{Definition}
\newtheorem*{observation*}{Observation}
\newtheorem*{remark*}{Remark}
\newtheorem*{example*}{Example}
\newtheorem*{question*}{Question}
\newtheorem*{exercise*}{Exercise}
\newtheorem*{fact*}{Fact}
\newtheorem*{notation*}{Notation}

\newcommand{\bbF}{\mathbb{F}}
\newcommand{\bbG}{\mathbb{G}}

\newcommand{\bbN}{\mathbb{N}}

\newcommand{\bbR}{\mathbb{R}}

\newcommand{\bbT}{\mathbb{T}}

\newcommand{\bbZ}{\mathbb{Z}}

\newcommand{\bfB}{\mathbf{B}}

\newcommand{\bfH}{\mathbf{H}}

\newcommand{\bfT}{\mathbf{T}}

\newcommand{\bfW}{\mathbf{W}}
\newcommand{\bfX}{\mathbf{X}}

\newcommand{\bfx}{\mathbf{x}}

\newcommand{\calE}{\mathcal{E}}

\newcommand{\calG}{\mathcal{G}}

\newcommand{\bbRnn}{\mathbb{R}_{\ge0}}

\newcommand{\cl}[1]{\overline{#1}}
\newcommand{\actson}{\curvearrowright}

\newcommand{\epi}{\twoheadrightarrow}

\newcommand{\normalin}{\lhd}

\newcommand{\ii}{^{-1}}

\newcommand{\gen}[1]{\left< #1 \right>}

\newcommand{\tild}[1]{\widetilde{#1}}

\DeclareMathOperator{\Aut}{Aut}

\DeclareMathOperator{\Out}{Out}

\DeclareMathOperator{\rank}{rank}

%% file: biblio.bib
@article {BFH:TitsII,
    AUTHOR = {Bestvina, Mladen and Feighn, Mark and Handel, Michael},
     TITLE = {The {T}its alternative for {${\rm Out}(F_n)$}. {II}. {A}
              {K}olchin type theorem},
   JOURNAL = {Ann. of Math. (2)},
  FJOURNAL = {Annals of Mathematics. Second Series},
    VOLUME = {161},
      YEAR = {2005},
    NUMBER = {1},
     PAGES = {1--59},
      ISSN = {0003-486X,1939-8980},
   MRCLASS = {20E36 (20E05 57M07)},
  MRNUMBER = {2150382},
MRREVIEWER = {S.\ Andreadakis},
       DOI = {10.4007/annals.2005.161.1},
       URL = {https://doi.org/10.4007/annals.2005.161.1},
}

@article {BestvinaHandel,
    AUTHOR = {Bestvina, Mladen and Handel, Michael},
     TITLE = {Train tracks and automorphisms of free groups},
   JOURNAL = {Ann. of Math. (2)},
  FJOURNAL = {Annals of Mathematics. Second Series},
    VOLUME = {135},
      YEAR = {1992},
    NUMBER = {1},
     PAGES = {1--51},
      ISSN = {0003-486X,1939-8980},
   MRCLASS = {20E05 (20F28 20F32 20F34 57M07 57M15)},
  MRNUMBER = {1147956},
MRREVIEWER = {Darryl\ McCullough},
       DOI = {10.2307/2946562},
       URL = {https://doi.org/10.2307/2946562},
}

@article{andrew2022free,
  title={Free-by-cyclic groups, automorphisms and actions on nearly canonical trees},
  author={Andrew, Naomi and Martino, Armando},
  journal={Journal of Algebra},
  volume={604},
  pages={451--495},
  year={2022},
  publisher={Elsevier}
}

@article{andrewguerchhorbezwade,
  title={Finiteness properties of centralisers in {$\text{Out}(F_n)$} and automorphisms of free-by-cyclic groups},
  author={Andrew, Naomi and Guerch, Yassine and Horbez, Camille and Wade, Richard D.},
  journal={in preparation},
  year={}
}

@article{yuwang,
  title={Covering degrees are determined by graph manifolds involved},
  author={Wang, Shicheng and Yu, Fengchun},
  journal={Commentarii Mathematici Helvetici},
  volume={74},
  number={2},
  pages={238--247},
  year={1999}
}

@article{sykiotis2018complexity,
  title={Complexity volumes of splittable groups},
  author={Sykiotis, Mihalis},
  journal={Journal of Algebra},
  volume={503},
  pages={409--432},
  year={2018},
  publisher={Elsevier}
}

@article {cohen1995verysmall,
    AUTHOR = {Cohen, Marshall M. and Lustig, Martin},
     TITLE = {Very small group actions on {${\bf R}$}-trees and {D}ehn twist
              automorphisms},
   JOURNAL = {Topology},
  FJOURNAL = {Topology. An International Journal of Mathematics},
    VOLUME = {34},
      YEAR = {1995},
    NUMBER = {3},
     PAGES = {575--617},
      ISSN = {0040-9383},
   MRCLASS = {20F32 (20E05 20E08 20F28)},
  MRNUMBER = {1341810},
MRREVIEWER = {Gilbert\ Levitt},
       DOI = {10.1016/0040-9383(94)00038-M},
       URL = {https://doi.org/10.1016/0040-9383(94)00038-M},
}

@article {Dahmani-Li,
    AUTHOR = {Dahmani, Fran\c{c}ois and Li, Ruoyu},
     TITLE = {Relative hyperbolicity for automorphisms of free products and
              free groups},
   JOURNAL = {J. Topol. Anal.},
  FJOURNAL = {Journal of Topology and Analysis},
    VOLUME = {14},
      YEAR = {2022},
    NUMBER = {1},
     PAGES = {55--92},
      ISSN = {1793-5253,1793-7167},
   MRCLASS = {20F67 (20F34 20F65)},
  MRNUMBER = {4411100},
MRREVIEWER = {Thomas\ Koberda},
       DOI = {10.1142/S1793525321500011},
       URL = {https://doi.org/10.1142/S1793525321500011},
}

@article {Ghosh,
    AUTHOR = {Ghosh, Pritam},
     TITLE = {Relative hyperbolicity of free-by-cyclic extensions},
   JOURNAL = {Compos. Math.},
  FJOURNAL = {Compositio Mathematica},
    VOLUME = {159},
      YEAR = {2023},
    NUMBER = {1},
     PAGES = {153--183},
      ISSN = {0010-437X,1570-5846},
   MRCLASS = {20F65 (20E05 20F67 57M07)},
  MRNUMBER = {4541452},
MRREVIEWER = {Anthony\ Genevois},
       DOI = {10.1112/S0010437X22007813},
       URL = {https://doi.org/10.1112/S0010437X22007813},
}

@article{guirardel2011trees,
  title={Trees of cylinders and canonical splittings},
  author={Guirardel, Vincent and Levitt, Gilbert},
  journal={Geometry \& Topology},
  volume={15},
  number={2},
  pages={977--1012},
  year={2011},
  publisher={Mathematical Sciences Publishers}
}

@article {lazarovichfir,
    AUTHOR = {Lazarovich, Nir},
     TITLE = {Finite index rigidity of hyperbolic groups},
   JOURNAL = {J. Eur. Math. Soc.},
  FJOURNAL = {Journal of European Mathematical Society},
      YEAR = {2025},
       DOI = {10.4171/JEMS/1618},
}

@article {lazarovichrahamimsisto,
 author = {Lazarovich, Nir and Rahamim, Gon and Sisto, Alessandro},
journal = {{P}reprint},
  title = {Finite Index Rigidity of Relatively Hyperbolic Groups},
   year = {2025},
   note = {\href{https://arxiv.org/abs/2509.04323}{arxiv:2509.04323}}
}

@incollection {lueckewu,
    AUTHOR = {Luecke, John and Wu, Ying-Qing},
     TITLE = {Relative {E}uler number and finite covers of graph manifolds},
 BOOKTITLE = {Geometric topology ({A}thens, {GA}, 1993)},
    SERIES = {AMS/IP Stud. Adv. Math.},
    VOLUME = {2.1},
     PAGES = {80--103},
 PUBLISHER = {Amer. Math. Soc., Providence, RI},
      YEAR = {1997},
      ISBN = {0-8218-0654-8},
   MRCLASS = {57N10 (57M10)},
  MRNUMBER = {1470722},
MRREVIEWER = {Bruno\ P.\ Zimmermann},
       DOI = {10.1090/amsip/002.1/05},
       URL = {https://doi.org/10.1090/amsip/002.1/05},
}

@article {JPPoly,
 author = {Mutanguha, Jean Pierre},
journal = {{P}reprint},
  title = {On polynomial free-by-cyclic groups},
   year = {2024},
   note = {\href{https://arxiv.org/abs/2412.16150}{arxiv:2412.16150}}
}

@article {Macura,
    AUTHOR = {Macura, Nata\v{s}a},
     TITLE = {Detour functions and quasi-isometries},
   JOURNAL = {Q. J. Math.},
  FJOURNAL = {The Quarterly Journal of Mathematics},
    VOLUME = {53},
      YEAR = {2002},
    NUMBER = {2},
     PAGES = {207--239},
      ISSN = {0033-5606,1464-3847},
   MRCLASS = {20F65 (20F69)},
  MRNUMBER = {1909513},
MRREVIEWER = {Athanase\ Papadopoulos},
       DOI = {10.1093/qjmath/53.2.207},
       URL = {https://doi.org/10.1093/qjmath/53.2.207},
}

@article {wangwu,
    AUTHOR = {Wang, Shi Cheng and Wu, Ying Qing},
     TITLE = {Covering invariants and co-{H}opficity of {$3$}-manifold
              groups},
   JOURNAL = {Proc. London Math. Soc. (3)},
  FJOURNAL = {Proceedings of the London Mathematical Society. Third Series},
    VOLUME = {68},
      YEAR = {1994},
    NUMBER = {1},
     PAGES = {203--224},
      ISSN = {0024-6115,1460-244X},
   MRCLASS = {57M10},
  MRNUMBER = {1243842},
MRREVIEWER = {Bruno\ P.\ Zimmermann},
       DOI = {10.1112/plms/s3-68.1.203},
       URL = {https://doi.org/10.1112/plms/s3-68.1.203},
}
